\documentclass{article}
\usepackage{graphicx} 
\usepackage{amsmath}

\title{On a question of Matt Baker regarding the dollar game}
\author{Marine Cases}
\date{April 2023}
\usepackage{xcolor}
\usepackage{tikz}

\usepackage{changepage}
\usepackage{amssymb}
\usepackage{amsthm}
\newtheorem{theorem}{Theorem}
\newtheorem{lemma}{Lemma}
\newtheorem{definition}{Definition}
\usepackage{hyperref}
\begin{document}
\setcounter{page}{1}
\maketitle
\definecolor{paleblue}{RGB}{192, 216, 230}
\definecolor{palepurple}{RGB}{216, 191, 216}
\definecolor{red}{RGB}{200, 50, 50} 
\definecolor{blue}{RGB}{50, 50, 200} 
\begin{abstract}
 In \cite{baker2010}, Matt Baker wrote the following about the dollar game played on a graph : “The total number of borrowing moves required to win the game when playing the “borrowing binge strategy” is independent of which borrowing moves you do in which order! Note, however, that it is usually possible to win in fewer moves by employing lending moves in combination with borrowing moves. The optimal strategy when one uses both kinds of moves is not yet understood.”
 
In this article we give a lower bound on the minimum number $M_{\min }$ of such moves of an optimal algorithm in term of the number of moves $M_0$ of the borrowing binge strategy. Concretely, we have : $M_{\min }\geq \dfrac{M_0}{n-1}$ where $n$ is the number of vertices of the graph. This bound is tight.
\end{abstract}
\vspace{0,5cm}
\tableofcontents

\newpage

\section{Introduction}
In this note, we consider the dollar game on a connected graph $G$ with initial divisor $C$, and focus on the minimum number of borrowing and lending moves needed to win the game (i.e, reach an all non-negative divisor). If the game is winnable, then it can always be won using only borrowing moves. More precisely, any winnable game can be won by simply using the borrowing binge strategy : any time there are vertices in debt, pick one of them and do a borrowing move. Repeat until everyone is out of debt!\\
This strategy is optimal if you are constrained to only use borrowing moves, but not if you can do both lending and borrowing moves, as seen below: \\

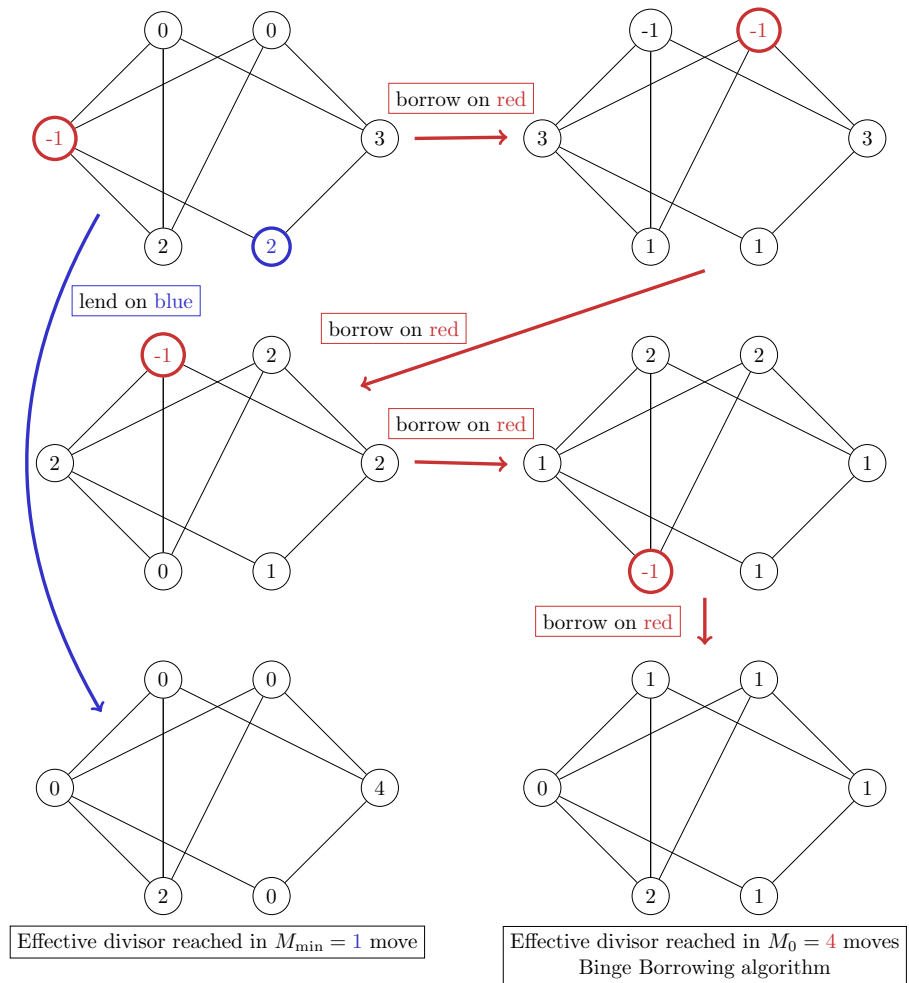
\begin{adjustwidth}{-.1cm}{0cm}
\scalebox{0.8}{\begin{tikzpicture}[every node/.style={circle, draw, minimum size=.5cm},scale=.9]
    \begin{scope}[local bounding box=scope1, shift={(0,0)}]
        \node[red, ultra thick] (v11) at (0,0) {-1};
        \node (v12) at (2,2) {0};
        \node (v13) at (2,-2) {2};
        \node (v14) at (4,2) {0};
        \node[blue, ultra thick]  (v15) at (4,-2) {2};
        \node (v16) at (6,0) {3};
        \draw (v11) -- (v12) -- (v13) -- (v11);
        \draw (v11) -- (v14);
        \draw (v11) -- (v15);
        \draw (v12) -- (v13);
        \draw (v12) -- (v16);
        \draw (v14) -- (v13);
        \draw (v14) -- (v16);
        \draw (v15) -- (v16);
    \end{scope}

    \begin{scope}[local bounding box=scope2, shift={(9,0)}]
        \node (v21) at (0,0) {3};
        \node (v22) at (2,2) {-1};
        \node (v23) at (2,-2) {1};
        \node[red, ultra thick] (v24) at (4,2) {-1};
        \node (v25) at (4,-2) {1};
        \node (v26) at (6,0) {3};
        \draw (v21) -- (v22) -- (v23) -- (v21);
        \draw (v21) -- (v24);
        \draw (v21) -- (v25);
        \draw (v22) -- (v23);
        \draw (v22) -- (v26);
        \draw (v24) -- (v23);
        \draw (v24) -- (v26);
        \draw (v25) -- (v26);
    \end{scope}

    \draw[->, ultra thick,red,text=black] (scope1.east) ++ (0.3,0)-- ([xshift=-0.3cm]scope2.west) node[midway, rectangle, thin, above=.4cm] {borrow on {\color{red} red}};

    \begin{scope}[local bounding box=scope3, shift={(0,-6)}]
        \node (v31) at (0,0) {2};
        \node[red, ultra thick] (v32) at (2,2) {-1};
        \node (v33) at (2,-2) {0};
        \node (v34) at (4,2) {2};
        \node (v35) at (4,-2) {1};
        \node (v36) at (6,0) {2};
        \draw (v31) -- (v32) -- (v33) -- (v31);
        \draw (v31) -- (v34);
        \draw (v31) -- (v35);
        \draw (v32) -- (v33);
        \draw (v32) -- (v36);
        \draw (v34) -- (v33);
        \draw (v34) -- (v36);
        \draw (v35) -- (v36);
    \end{scope}
\draw[->, ultra thick,red,text=black] (scope2.south) ++ (0,-.1)-- ([xshift=-0.7cm,yshift=1.4cm,text=black]scope3.east) node[midway, rectangle, thin, left=1.1cm] {borrow on {\color{red} red}};
    \begin{scope}[local bounding box=scope4, shift={(9,-6)}]
        \node (v41) at (0,0) {1};
        \node (v42) at (2,2) {2};
        \node[red, ultra thick] (v43) at (2,-2) {-1};
        \node (v44) at (4,2) {2};
        \node (v45) at (4,-2) {1};
        \node (v46) at (6,0) {1};
        \draw (v41) -- (v42) -- (v43) -- (v41);
        \draw (v41) -- (v44);
        \draw (v41) -- (v45);
        \draw (v42) -- (v43);
        \draw (v42) -- (v46);
        \draw (v44) -- (v43);
        \draw (v44) -- (v46);
        \draw (v45) -- (v46);
    \end{scope}

     \draw[->, ultra thick,red,text=black] (scope3.east) ++ (0.3,0)-- ([xshift=-0.3cm]scope4.west) node[midway, rectangle, thin, above=.4cm] {borrow on {\color{red} red}};
    \begin{scope}[local bounding box=scope5, shift={(9,-12)}]
        \node (v51) at (0,0) {0};
        \node (v52) at (2,2) {1};
        \node (v53) at (2,-2) {2};
        \node (v54) at (4,2) {1};
        \node (v55) at (4,-2) {1};
        \node (v56) at (6,0) {1};
        \draw (v51) -- (v52) -- (v53) -- (v51);
        \draw (v51) -- (v54);
        \draw (v51) -- (v55);
        \draw (v52) -- (v53);
        \draw (v52) -- (v56);
        \draw (v54) -- (v53);
        \draw (v54) -- (v56);
        \draw (v55) -- (v56);
    \end{scope}

   \draw[->, ultra thick,red,text=black] (scope4.south) ++ (0,-.1)-- ([yshift=.3cm]scope5.north) node[midway, rectangle, thin, left=.4cm] {borrow on {\color{red} red}};

    \begin{scope}[local bounding box=scope2, shift={(0,-12)}]
        \node (v21) at (0,0) {0};
        \node (v22) at (2,2) {0};
        \node (v23) at (2,-2) {2};
        \node (v24) at (4,2) {0};
        \node (v25) at (4,-2) {0};
        \node (v26) at (6,0) {4};
        \draw (v21) -- (v22) -- (v23) -- (v21);
        \draw (v21) -- (v24);
        \draw (v21) -- (v25);
        \draw (v22) -- (v23);
        \draw (v22) -- (v26);
        \draw (v24) -- (v23);
        \draw (v24) -- (v26);
        \draw (v25) -- (v26);
    \end{scope}
       \draw[->, ultra thick,blue] ([xshift=1.2cm, yshift=-1.4cm]scope1.west)  to[out=180, in=180, bend right]([xshift=1.2cm,yshift=1.4cm]scope2.west);
       \node[rectangle, align=left,blue,,text=black, thin] at (1.5,-3) {lend on {\color{blue} blue}};
        \node[rectangle, thin,] at (3,-14.85) {Effective divisor reached in $M_{\min }={\color{blue}1}$ move};
         \node[rectangle, thin,align=center] at (12,-15.1) {Effective divisor reached in $M_0={\color{red} 4}$ moves\\ Binge Borrowing algorithm};
\end{tikzpicture}}
\end{adjustwidth}
In this article we will prove a bound $M_{\min }\geq \dfrac{M_0}{n-1}$ on how far from optimal the Borrowing Binge strategy can be, where n is the number of vertices of G.

We start by reinterpreting the question in terms of the chip firing game through a well-known bijection $v_{0}\rightarrow K-v_{0}$ (See \cite{book},
p 40), where $K=(d_i-1)_i$ is the maximal stable divisor.

In this context, the borrowing binge strategy becomes the standard greedy chip-firing game : any time there are vertices with at least as many chips as their degree, pick one of them and do a lending move; repeat until the divisor is stable and no move can be made.
The effective divisor then becomes a stable divisor.
Then the question becomes  what is the minimal number of lending and borrowing moves to reach a stable divisor.

The proof of our main theorem $M_{\min }\geq \dfrac{M_0}{n-1}$, with $M_{\min }$ the minimum number of moves of an optimal algorithm and $M_0$ the number of moves  of the borrowing binge strategy, relies on two lemmas : 

\begin{enumerate}
    \item The principle of least action which says that the greedy chip firing game strategy is optimal to reach a stable divisor if we restrict to only lending moves.
    \item  A lemma which tells us by performing only lending moves, we can at most cut the length by a multiplicative factor $\dfrac{1}{n-1}$ if we allow borrowing moves as well.
\end{enumerate}
Combining the two lemmas give the theorem.

\subsection{Definitions}

First, let us introduce some basic definitions.
We will consider only simple connected graphs throughout this article.
\begin{definition}
\begin{itemize}
    \item The Laplacian matrix of G, denoted $\Delta$, is a matrix that captures the structure of a graph. Its elements are given by:
    \[
    \Delta_{ij} =
    \begin{cases}
      d_i & \text{if } i=j. \\
      -1 & \text{if } i\neq j \text{ and vertices } i \text{ and } j \text{ are adjacent.} \\
      0 & \text{otherwise.}
    \end{cases}
    \]
    where $d_i$ is the degree of the vertex $i$.
    \item A \textbf{divisor} is a vector in $\mathbf{Z}^n$ that represents the number of chips on every vertex of the graph.
    \item Given two divisors $C$ and $C'$, and $v \in \mathbf{Z}^n$, the notation $C \xrightarrow{v} C'$ means that $C'=C-\Delta \cdot v$.
    \item Let $C$ be a divisor:
    \begin{itemize}
        \item A \textbf{lending move} on vertex $i$ corresponds to $C\xrightarrow{e_i}C'$ and means that the vertex $i$ lends one chip to each of its neighbors.
        \item A \textbf{borrowing move} on vertex $i$ corresponds to $C\xrightarrow{-e_i}C'$ and means that the vertex $i$ borrows one chip from each of its neighbors.
    \end{itemize}
    \item An \textbf{effective divisor} $C$ is a divisor with $C\geq 0$.
    \item A \textbf{stable divisor} $C$ is a divisor with $C\leq K=(d_i-1)_i$, meaning that one cannot perform a lending move without becoming indebted, i.e $(C')_i<0$ for some i .
    \item A \textbf{closest stable divisor} $C^*$ to a given divisor $C$ is defined as a stable divisor such that the number of moves (either borrowing or lending) required to reach any other stable divisor from $C$ is no less than the number of moves required to reach $C^*$ from $C$. 
    \item The \textbf{distance} between two divisors is the minimal number of moves required to reach one from the other.  
\end{itemize}
\end{definition}

Let $C$ be a divisor and consider $C'=K-C$. Then a borrowing move if a vertex $i$ has a non positive number of chips in $C$ corresponds to a lending move if a vertex $i$ has at least $ d_{i}$ chips.

We have then the following equivalences:
\begin{center}
\begin{tabular}{|p{5.8cm}|p{5.8cm}|}
\hline\\
\text{Binge strategy for dollar game} & \text{Chip firing game}\\
\hline
\text{Effective divisor} & \text{Stable divisor}\\
\hline
Borrow on negative vertices & Lend on vertices with at least as many chips as the degree\\
\hline
 Borrowing binge strategy to reach an effective divisor & Standard greedy chip firing game (without a sink) to reach a stable divisor\\ 
\hline
Riemann-Roch formula (see \cite{baker_norine}): number of chips that can be removed before no longer having an effective divisor in the class  (or, equivalently, the borrowing binge is no longer terminating) & Riemann-Roch formula: number of chips that can be added before no longer having stable divisor in the class (or, equivalently, the game is no longer terminating)\\
\hline
\end{tabular}
\end{center}

\section{Main Theorem}

\begin{theorem}

Let $C\xrightarrow{v_0} X_{0}$, where $v_0$ is the firing sequence found with the greedy algorithm. 
If $\left| v_{0}\right|_{1} = M_0$
then the closest stable divisor must be at distance at least $\dfrac{M_0}{n-1}$.
\end{theorem}
\noindent The two examples afterwards prove that the bound in Theorem 1 is tight. 
\newline

The following is a classic lemma from chip-firing game theory :
\begin{lemma}
(Least action principle.)
Let $C\xrightarrow{v_0} X_{0}$ be the sequence from the greedy algorithm of the classical chip-firing game. Then, if $C\xrightarrow{v_1} X_{1}$ is another sequence to another stable configuration, we have that if $v_{1}\geq 0$, then $v_{1}\geq v_{0}$. (For a proof of this lemma, see \cite{book}, pp.39-40)
\end{lemma}

\begin{lemma}
Let $v\geq 0$ be a vector with at least one zero coordinate such that $C\xrightarrow{v} X$. Then the shortest path $v'$ from $C$ to $X$ satisfies $\left| v'\right|_{1}\geq \dfrac{\left| v\right|_{1}}{n-1}$.
\end{lemma}

\begin{proof}
In this proof, $k/2$ mean the integer $\lfloor k/2\rfloor$, for every integer $k$.

Since $\mathbf{1}=\begin{pmatrix}
1 \\
\vdots  \\
1
\end{pmatrix}$ is a basis for $\text{ker }\Delta$, we are looking for $v_{min}=v + k\cdot \mathbf{1}, k \in \mathbf{Z}$ with minimal $L^1$-norm.

First, note that if $v'$ has $m$ positive coordinates, $l$ zero coordinates, and $r=n-m-l$ negative coordinates, then 

$$\begin{aligned}\left| v'-1\cdot \mathbf{1}\right| _{1}&=\left| v'\right|_{1}+ r+l -m&\\
&=\left| v'\right| _{1}+n-2m &\\
 &<\left| v'\right| _{1}& \text{if } m>\dfrac{n}{2}\end{aligned}
$$

and

$$\begin{aligned}\left| v'+1\cdot \mathbf{1}\right| _{1}&=\left| v'\right|_{1}-r+l +m&\\
&=\left| v'\right|_{1}+n-2r &\\
 &<\left| v'\right|_{1}& \text{if }  r>\dfrac{n}{2}. \end{aligned}
$$
Thus the minimum can be obtained only when the number of positive and negative components are both at most $n/2$. Therefore, we must remove $k\cdot \mathbf{1}$ from $v$, with $k$ chosen so that the new vector $v_{min}=v-k\cdot \mathbf{1}$ has exactly $n/2$ positive or zero components and $n/2$ negative or zero components. In this case, the $L^1$ norm of this vector will be minimal. Concretely :
\begin{itemize}
\item For $n$ even, the vector with $\frac{n}{2}-1$ positive components, $\frac{n}{2}$ negative components and 1 zero component is a minimum.
\item For $n$ odd, the minimum is unique. The vector with $\frac{n}{2}$ positive and negative components, and 1 zero component is the minimun.
\end{itemize}

Both cases can thus be dealt with simultaneously by writing: $\frac{n-1}{2}$ positive components, $\frac{n}{2}$ negative components and 1 zero component.

Let us now compute this minimal norm. 
We want to show that $\left| v_{\min }\right| _{1}\geq \dfrac{\left| v\right| _{1}}{n-1}$.

Without loss of generality, let us assume that $v$ satifies $$v=\begin{pmatrix}
x_{1} \\
\vdots\\
x_{n}
\end{pmatrix} \quad 0=x_{1}\leq \ldots \leq x_{n}.$$ (We sort its components for simplicity.)

Then $k=x_{\frac{n-1}{2}}$  and 
$$v_{min}=\begin{pmatrix}
0 \\
x_{2} \\
\vdots  \\
x_{n}
\end{pmatrix}-x_{\frac{n-1}{2}}\cdot \begin{pmatrix}
1 \\
\vdots  \\
1
\end{pmatrix}=\begin{pmatrix}
-x_{\frac{n-1}{2}} \\
x_{2}-x_{\frac{n-1}{2}} \\
\vdots  \\
x_{n}-x_{\frac{n-1}{2}}
\end{pmatrix}.
$$
Thus
\begin{align*}
    \left|\begin{pmatrix}
-x_{\frac{n-1}{2}} \\
x_{2}-x_{\frac{n-1}{2}} \\
\vdots  \\
x_{n}-x_{\frac{n-1}{2}}
\end{pmatrix}\right| _{1}&= x_{\frac{n-1}{2}} +\sum^{\frac{n-1}{2}-1}_{i=2}x_{\frac{n-1}{2}} -x_{i}+\sum ^{n}_{i=\frac{n-1}{2}+1}x_{i}-x_{\frac{n}{2}}\\
&=x_{\frac{n-1}{2}} +\underbrace{\sum^{\frac{n-1}{2}-1}_{i=2}x_{\frac{n-1}{2}} -x_{i}}_{\geq 0}+\underbrace{\sum ^{n-1}_{i=\frac{n-1}{2}+1}x_{i}-x_{\frac{n-1}{2}}}_{\geq 0}+(x_{n-1}-x_{\frac{n-1}{2}})\\
&\geq x_{n} \\
&\geq \dfrac{\sum ^{n}_{i=2}x_{i}}{n-1}=\dfrac{\sum ^{n-1}_{i=1}x_{i}}{n-1}\text{ since $x_1=0$}\\
&=\dfrac{\left| v\right| _{1}}{n-1}.
\end{align*}

\end{proof}

Using these two lemmas, we can show our theorem: 
\begin{proof}
    
Let $C\xrightarrow{v_0} X_{0}$ be the sequence from the greedy algorithm. We want to find a closest stable divisor. Let $v_{1}$ be a sequence from $C$ to a closest stable divisor. Since $\begin{pmatrix}
1 \\
\vdots  \\
1
\end{pmatrix}$ is a basis for $\text{ker }\Delta$, we can always choose $v_{1}\geq 0$ with one zero coordinate. By Lemma 2, we have $v_{1}\geq v_{0}\geq 0$.

Now, by Lemma 1, the shortest path from C to the stable divisor that $v_{1}$ reaches is $v_{1}^{'}$, with $\left| v_{1}'\right|_{1}\geq \dfrac{\left| v_{1}\right|_{1}}{n-1}$.

Moreover, since $v_{1}\geq v_{0} \geq 0$ , we have $M_{min}=\left| v_{1}'\right|_{1}\geq \dfrac{\left| v_{1}\right|_{1}}{n-1}\geq \dfrac{\left| v_{0}\right|_{1}}{n-1}=\dfrac{M_0}{n-1}$.
\end{proof}

\subsection{Examples where the bound in Theorem 1 is optimal }
In this section, we will construct two examples to show that the bound we obtained is tight.\\
\textbf{Example 1 :}

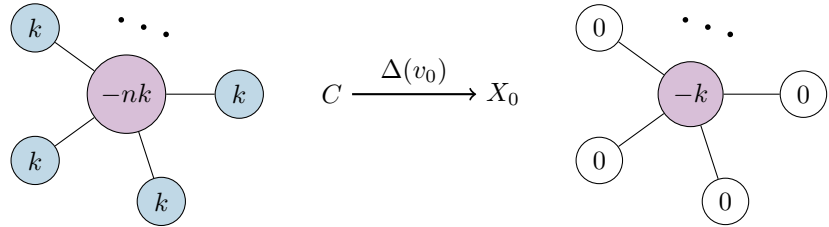
\begin{tikzpicture}[scale=1.5]
\begin{scope}[shift={(0,0)}] 
    \node[circle,draw, fill=palepurple](Z1) at (0,0) {$-nk$};

    \foreach \i in {1,3,4,5} {
        \node[circle,draw, fill=paleblue](A\i) at ({360/5*(\i-1)}:1) {$k$};
        \draw (Z1)--(A\i);
    }

    \node[scale=2,rotate=-18] at ({360/5*(1)}:.6) {$\cdots$};
\end{scope}

\begin{scope}[shift={(5,0)}] 
    \node[circle,draw, fill=palepurple](Z2) at (0,0) {$-k$};

    \foreach \i in {1,3,4,5} {
        \node[circle,draw](B\i) at ({360/5*(\i-1)}:1) {$0$};
        \draw (Z2)--(B\i);
    }

    \node[scale=2,rotate=-18] at ({360/5*(1)}:.6) {$\cdots$};
\end{scope}

\draw[->, thick] (2,0) node[left] {$C$} -- (3.1,0) node[midway,above] {$\Delta(v_0)$} node[right] {$X_0$};

\end{tikzpicture}\\[0,2cm]

Let $C\xrightarrow{v_0} X_{0}$ with 
$v_{0}=\begin{pmatrix}
k \\
\vdots \\
k \\
0 \\
k \\
\vdots \\
k
\end{pmatrix}.$
We can choose $v_{min}=v_{0}-k \cdot\begin{pmatrix}
1 \\
1 \\
\vdots  \\
1 \\
1 \\
1
\end{pmatrix}
=\begin{pmatrix}
0 \\
\vdots \\
0 \\
-k \\
0 \\
\vdots \\
0 \\
\end{pmatrix}.$

\noindent Note here that we didn’t even need to find another stable divisor : the same divisor with another (shorter but equivalent) vector of firings give us the best possible result.\\

\textbf{Example 2: }
The preceding example offers an optimal, linear (in $n$) improvement. One might think that this is the only way to achieve the optimal bound, through an equivalent and much shorter path to the same divisor, via the following algorithm : compute the greedy path, identify the vector with the smallest norm in $\mathbf{coker}(\Delta)$, and use this ‘renormalization’ to obtain a linear improvement. This suggests that, once we take into account the ‘renormalized’ path to the greedy divisor, we might no longer obtain a linear improvement by finding a path to a closer stable divisor. However, our second example demonstrates that this is not the case. More precisely, we prove that we can achieve an improvement by a factor of $n/2$ by reaching a stable divisor entirely different from the greedy one, and this factor of $n/2$ holds even when we account for the ‘renormalized’ path to the greedy divisor.\\

\noindent $n$ will denote an even integer in the following.\\

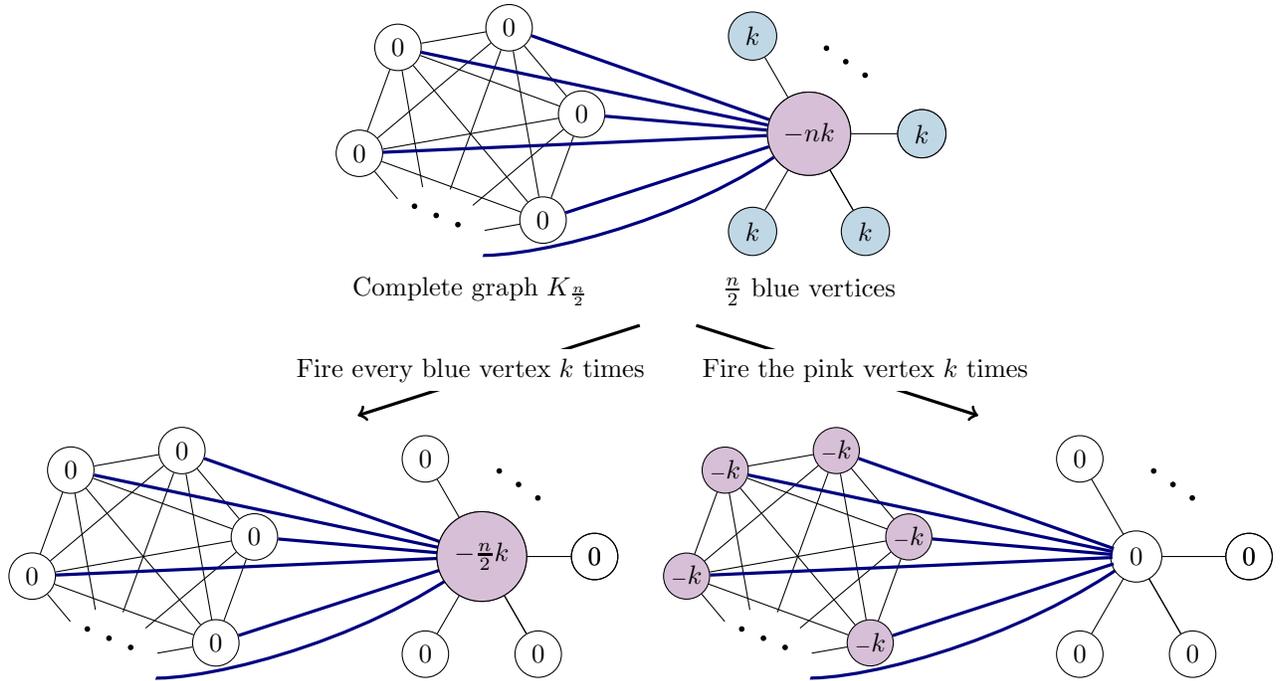
\definecolor{paleblue}{RGB}{192, 216, 230}
\definecolor{palepurple}{RGB}{216, 191, 216}
\definecolor{myblue}{RGB}{0,0,128}
\begin{adjustwidth}{-2.2cm}{-2.2cm}
\begin{tikzpicture}[scale=1.5]

\begin{scope}[shift={(0,0)}]
    \node[circle,draw, fill=palepurple,inner sep=4pt](Z1) at (0,0) {$-nk$};

    \foreach \i in {1,...,3} {
        \node[circle,draw, fill=paleblue](A\i) at ({360/3*(\i-1)}:1) {$k$};
        \draw (Z1)--(A\i);
    }
    
    \node[circle,draw, fill=paleblue](A4) at ({360/3*(2)+60}:1) {$k$};
        \draw (Z1)--(A4);
     \node[rotate=-35,scale=2] at ({360/3-60}:.7) {$\cdots$}; 
    \node[circle,draw, fill=paleblue](A5) at ({360/3*3}:1) {$k$};
    \draw (Z1)--(A4);
\end{scope}

\begin{scope}[shift={(-3,0)}]
    \foreach \i in {1,...,6} {
        \node[circle,draw](B\i) at ({70+360/6*(\i-1)}:1) {$0$};
    }
  \foreach \i in {1,...,5} { 
    \pgfmathsetmacro{\startval}{int(\i+1)}
    \foreach \j in {\startval,...,6} {
        \draw (B\i) -- (B\j);
    }
}
    
    \end{scope}
\draw[very thick,myblue] (Z1) .. controls (-1.3049, -0.8749) and (-2.86, -1.20) .. (B4);
\node[circle,draw, fill=palepurple,inner sep=4pt] at (0,0) {$-nk$};

\begin{scope}[shift={(-3,0)}]
\node[circle,draw,color=white,fill=white,thick] at ({70+360/6*(4-1)}:1) {$-nkkk$};
\node[rotate=-23,scale=2] at ({70+360/6*(4-1)}:.8) {$\cdots$};

\end{scope}

\foreach \i in {1,2,3,5,6} {
    \draw[very thick,myblue] (Z1)-- (B\i);
}
\node at (-3,-1.4) {Complete graph $K_\frac{n}{2}$};
\node at (0,-1.4) {$\frac{n}{2}$ blue vertices};
  \draw[->, very thick]
    (-1.5,-1.7) -- (-4,-2.5);
    \node[fill=white] at (-3,-2.1) {Fire every blue vertex $k$ times };

      \draw[->, very thick]
    (-1,-1.7) -- (1.5,-2.5);
    \node[fill=white] at (.5,-2.1) {Fire the pink vertex $k$ times };

\begin{scope}[shift={(-2.9,-3.75)}]
\begin{scope}[shift={(0,0)}]
    \node[circle,draw, ](Z1) at (0,0) {$\frac{n}{2}k$};

    \foreach \i in {1,...,3} {
        \node[circle,draw](A\i) at ({360/3*(\i-1)}:1) {$0$};
        \draw (Z1)--(A\i);
    }
    
    \node[circle,draw, ](A4) at ({360/3*(2)+60}:1) {$0$};
        \draw (Z1)--(A4);
     \node[rotate=-35,scale=2] at ({360/3-60}:.7) {$\cdots$}; 
    \node[circle,draw](A5) at ({360/3*3}:1) {0};
    \draw (Z1)--(A4);
\end{scope}

\begin{scope}[shift={(-3,0)}]
    \foreach \i in {1,...,6} {
        \node[circle,draw](B\i) at ({70+360/6*(\i-1)}:1) {$0$};
    }
    \foreach \i in {1,...,5} { 
    \pgfmathsetmacro{\startval}{int(\i+1)}
    \foreach \j in {\startval,...,6} {
        \draw (B\i) -- (B\j);
    }
}

    \end{scope}
\draw[very thick,myblue] (Z1) .. controls (-1.3049, -0.8749) and (-2.86, -1.20) .. (B4);

\begin{scope}[shift={(-3,0)}]
\node[circle,draw,color=white,fill=white,thick] at ({70+360/6*(4-1)}:1) {$-nkkk$};
\node[rotate=-23,scale=2] at ({70+360/6*(4-1)}:.8) {$\cdots$};

\end{scope}

\foreach \i in {1,2,3,5,6} {
    \draw[very thick,myblue] (Z1)-- (B\i);
}
    \node[circle,draw,inner sep=4pt,fill=white] at (0,0) {$-\frac{n}{2}k$};
\node[circle,draw,inner sep=4pt,fill=palepurple] at (0,0) {$-\frac{n}{2}k$};
\end{scope}

\begin{scope}[shift={(2.9,-3.75)}]
\begin{scope}[shift={(0,0)}]
    \node[circle,draw, ](Z1) at (0,0) {$0$};

    \foreach \i in {1,...,3} {
        \node[circle,draw](A\i) at ({360/3*(\i-1)}:1) {$0$};
        \draw (Z1)--(A\i);
    }
    
    \node[circle,draw, ](A4) at ({360/3*(2)+60}:1) {$0$};
        \draw (Z1)--(A4);
     \node[rotate=-35,scale=2] at ({360/3-60}:.7) {$\cdots$}; 
    \node[circle,draw](A5) at ({360/3*3}:1) {0};
    \draw (Z1)--(A4);
\end{scope}

\begin{scope}[shift={(-3,0)}]
    \foreach \i in {1,...,6} {
        \node[circle,draw,fill=palepurple,inner sep=1.2pt](B\i) at ({70+360/6*(\i-1)}:1) {${\scriptstyle-} k$};
    }
  \foreach \i in {1,...,5} { 
    \pgfmathsetmacro{\startval}{int(\i+1)}
    \foreach \j in {\startval,...,6} {
        \draw (B\i) -- (B\j);
    }
}
    
    \end{scope}
\draw[very thick,myblue] (Z1) .. controls (-1.3049, -0.8749) and (-2.86, -1.20) .. (B4);
\node[circle,draw,inner sep=4pt,fill=white] at (0,0) {0};
\begin{scope}[shift={(-3,0)}]
\node[circle,draw,color=white,fill=white,thick] at ({70+360/6*(4-1)}:1) {$-nkkk$};
\node[rotate=-23,scale=2] at ({70+360/6*(4-1)}:.8) {$\cdots$};

\end{scope}

\foreach \i in {1,2,3,5,6} {
    \draw[very thick,myblue] (Z1)-- (B\i);
}
    
\end{scope}

\end{tikzpicture}
\end{adjustwidth}

\noindent Let $C \xrightarrow{\Delta{(v_0)}} X_0$ with 
$v_{0}=\begin{pmatrix}
k \\
\vdots \\
k \\
0 \\
0 \\
\vdots \\
0
\end{pmatrix}.$
Here $v_0$ is the best possible among all vectors in $v_0+ker \Delta$, since there are at most $n/2$ positive and negative components (cf lemma 2).
Thus, to improve, we must find a sequence to another stable divisor. Note that the vector $v_{min}
=\begin{pmatrix}
0 \\
\vdots \\
0 \\
-k \\
0 \\
\vdots \\
0 \\
\end{pmatrix}$
satisfies $C \xrightarrow{\Delta v_{\text{min}}} X_{1}$
and we have $\left| V_{\min }\right| _{1}=\dfrac{\left| v_{0}\right| _{1}}{\frac{n}{2}}$.\\
\section{Acknowledgements}
I would like to thank Matt Baker for his advice (both on a mathematical and a non-mathematical level) and for suggesting this problem to me.

\bibliographystyle{plainurl}  

\bibliography{Main}

\end{document}